\documentclass[12pt]{amsart}
\usepackage{graphicx}
\usepackage[active]{srcltx}
\usepackage{amsthm,mathrsfs}
\usepackage{a4wide}
\usepackage[english]{babel}
\usepackage{amsfonts}
\usepackage{amsmath}
\usepackage{amssymb}
\usepackage{dsfont}
\newtheorem{lemma}{Lemma}
\newtheorem{theorem}{Theorem}
\newtheorem{corollary}{Corollary}

\usepackage{float}
\usepackage{url}

\newcommand {\E} {\mathbb{E}}

\newcommand {\p} {\mathbb{P}}

\newcommand {\Hy} {\mathcal{H}}

\newcommand {\ed} {\mathcal{E}}
\newcommand {\R} {\mathbb{R}}

\DeclareMathOperator {\disc} {disc}
\DeclareMathOperator {\lindisc} {lindisc}
\DeclareMathOperator {\herdisc} {herdisc}

\newcommand {\ve} {\varepsilon}

\makeatletter\def\blfootnote{\xdef\@thefnmark{}\@footnotetext}\makeatother
\allowdisplaybreaks
\title{\bf Low-discrepancy point sets for non-uniform measures}

\author{Christoph Aistleitner} 
\address{School of Mathematics and Statistics, University of New South Wales, Sydney NSW 2052, Australia}
\email{aistleitner@math.tugraz.at}

\author{Josef Dick} 
\address{School of Mathematics and Statistics, University of New South Wales, Sydney NSW 2052, Australia}
\email{josef.dick@unsw.edu.au}

\thanks{The first author is supported by a Schr\"odinger scholarship of the Austrian Research
Foundation (FWF). The second author is supported by an Australian Research Council Queen Elizabeth 2 Fellowship.}


\begin{document}

\begin{abstract}
In the present paper we prove several results concerning the existence of low-discrepancy point sets with respect to an arbitrary non-uniform measure $\mu$ on the $d$-dimensional unit cube. We improve a theorem of Beck, by showing that for any $d \geq 1$, $N \geq 1,$ and any non-negative, normalized Borel measure $\mu$ on $[0,1]^d$ there exists a point set $x_1, \dots, x_N \in [0,1]^d$ whose star-discrepancy with respect to $\mu$ is of order
$$
D_N^*(x_1, \dots, x_N; \mu) \ll \frac{(\log N)^{(3d+1)/2}}{N}.
$$
For the proof we use a theorem of Banaszczyk concerning the balancing of vectors, which implies an upper bound for the linear discrepancy of hypergraphs. Furthermore, the theory of large deviation bounds for empirical processes indexed by sets is discussed, and we prove a numerically explicit upper bound for the inverse of the discrepancy for Vapnik--\v{C}ervonenkis classes. Finally, using a recent version of the Koksma--Hlawka inequality due to Brandolini, Colzani, Gigante and Travaglini, we show that our results imply the existence of cubature rules yielding fast convergence rates for the numerical integration of functions having discontinuities of a certain form.
\end{abstract}

\date{}
\maketitle

\section{Introduction} \label{sec1}

Let $x_1, \dots, x_N$ be a set of points in the $d$-dimensional unit cube $[0,1]^d$. The \emph{star-discrepancy} $D_N^*$ of $x_1, \dots, x_N$ is defined as
\begin{equation} \label{dn}
D_N^*(x_1, \dots, x_N) = \sup_{A \subset [0,1]^d} \left| \frac{1}{N} \sum_{n=1}^N \mathds{1}_A (x_n) - \lambda(A) \right|,
\end{equation}
where $\mathds{1}_A$ denotes the indicator function of $A$, $\lambda$ is the $d$-dimensional Lebesgue measure, and the supremum is extended over all axis-parallel boxes $A$ which have one vertex at the origin. By the Koksma--Hlawka inequality, for any $x_1, \dots, x_N$ and any function $f$ which has total variation $\textup{Var}~f$ on $[0,1]^d$ (in the sense of Hardy and Krause) we have
\begin{equation} \label{kokshla}
\left| \frac{1}{N}\sum_{n=1}^N f(x_n) - \int_{[0,1]^d} f(x) ~dx \right| \leq D_N^*(x_1, \dots, x_N) \cdot \textup{Var}~f.
\end{equation}
Consequently, point sets having small discrepancy can be used for numerical integration. This method is called \emph{Quasi-Monte Carlo method} (as opposed to the \emph{Monte Carlo method}, which uses randomly sampled points), and is heavily used in applications such as option pricing in financial mathematics. There exist several constructions of $d$-dimensional $N$-element point sets achieving a discrepancy of order 
\begin{equation} \label{lowd}
D_N^*(x_1, \dots, x_N) \ll \frac{(\log N)^{d-1}}{N}.
\end{equation}
Consequently, the error of Quasi-Monte Carlo integration for cleverly chosen sampling points is asymptotically significantly smaller than the (probabilistic) error of order $N^{-1/2}$ of the Monte Carlo method. For more information on discrepancy theory and the Quasi-Monte Carlo method, see for example~\cite{dpd,dts,knu}.\\

Upper bounds of the form~\eqref{lowd} are only useful if the number of points $N$ is large in comparison with the dimension $d$; in particular, the expression on the right-hand side of~\eqref{lowd} is increasing for $N \leq e^{d-1}$. To investigate low-discrepancy point sets whose cardinality is of moderate size in comparison with the dimension $d$, the notion of the \emph{inverse of the star-discrepancy} was introduced. Let $n^*(d,\ve)$ denote the smallest possible cardinality of a point set in $[0,1]^d$ achieving a star-discrepancy of at most $\ve$. Heinrich,  Novak, Wasilkowski and Wo{\'z}niakowski~\cite{hnww} proved that
\begin{equation} \label{hnwwb}
n^*(d,\ve) \leq c_{\textup{abs}} d \ve^{-2}
\end{equation}
for some absolute constant $c_{\textup{abs}}$, and Hinrichs~\cite{hinr} obtained the lower bound 
\begin{equation} \label{hin}
n^*(d,\ve) \geq c_{\textup{abs}} d \ve^{-1}.
\end{equation}
Thus the inverse of the star-discrepancy depends linearly on the dimension; on the other hand, the precise dependence of $n^*(d,\ve)$ on $\ve$ is still an open problem.\\

The notion of the star-discrepancy can be easily generalized to measures $\mu$ different from the Lebesgue measure $\lambda$ (which is the measure of the uniform distribution on $[0,1]^d$). In the following we will assume that $\mu$ is a real, non-negative, normalized Borel measure on $[0,1]^d$. Similar to~\eqref{dn}, we define the star-discrepancy of a point set $x_1, \dots, x_N \in [0,1]^d$ with respect to $\mu$ by
\begin{equation} \label{discmu}
D_N^*(x_1, \dots, x_N;\mu) = \sup_{A \subset [0,1]^d} \left| \frac{1}{N} \sum_{n=1}^N \mathds{1}_A (x_n) - \mu(A) \right|.
\end{equation}
For any such $\mu$, any $d \geq 1$ and any $N \geq 1$, Beck~\cite{beck} proved the existence of a point set $x_1, \dots, x_N \in [0,1]^d$ for which
\begin{equation} \label{bec}
D_N^*(x_1, \dots, x_N;\mu) \ll \frac{\left( \log N \right)^{2d}}{N},
\end{equation} 
where the (unspecified) implied constant may only depend on $d$, but not on $N$ and, somewhat surprisingly, also not on the measure $\mu$. As a consequence for any $\mu$ and any $d\geq 1$ there also exists an infinite sequence of points $(x_n)_{n \geq 1}$ from $[0,1]^d$ for which
$$
D_N^*(x_1, \dots, x_N;\mu) \ll \frac{\left( \log N \right)^{2d+2}}{N} \qquad \textrm{for all $N \geq 1$}.
$$

The purpose of the present paper is to obtain an improved and numerically explicit version of Beck's theorems, using results on the balancing of vectors and the discrepancy of hypergraphs (these concepts are described in detail in Section~\ref{sechy} below). For the proof we need to approximate a general measure by a discrete measure, and we show how this problem is connected with the theory of large deviations of empirical processes and the inverse of the star-discrepancy. We also prove a general, numerically explicit version of~\eqref{hnwwb} (see Section~\ref{sec2}). Finally, combining our results with a recent version of the Koksma--Hlawka inequality due to Brandolini \emph{et al.}~\cite{bcgt}, we prove the existence of cubature rules achieving fast asymptotic convergence rates for the numerical integration of functions having discontinuities of a certain form (see Section~\ref{sec3}).\\

\begin{theorem} \label{th1}
Let $\mu$ be any non-negative normalized Borel measure on $[0,1]^d$, where $d \geq 1$ is arbitrary. Then for any $N \geq 1$ there exist points $x_1, \dots, x_N \in [0,1]^d$ such that
$$
D_N^*(x_1, \dots, x_N;\mu) \leq 63 \sqrt{d} \frac{ \left(2 + \log_2 N \right)^{(3d+1)/2}}{N}.
$$
\end{theorem}

\begin{theorem} \label{th2}
Let $\mu$ be any non-negative normalized Borel measure on $[0,1]^d$, where $d \geq 1$ is arbitrary. Then there exists an infinite sequence $(x_n)_{n \geq 1}$ of points from $[0,1]^d$ such that
$$
D_N^*(x_1, \dots, x_N; \mu) \leq 133 \sqrt{d+1} \frac{(4 + 2 \log_2 N)^{(3d+4)/2}}{N} \qquad \textrm{for $N \geq 1$.}
$$
\end{theorem}
In the statement of these theorems and throughout the rest of this paper $\log$ denotes the natural logarithm, and $\log_2$ denotes the dyadic logarithm.\\

For the proofs of these theorems we need several auxiliary results concerning the linear discrepancy of hypergraphs and the approximation of measures by discrete measures; they are presented in Sections \ref{sechy} and \ref{sec2}, respectively. In Section \ref{sec3} we present an application of Theorem~\ref{th1} for the numerical integration of certain discontinuous functions. Some open problems are stated in Section~\ref{secop}, and proofs are given in Section~\ref{secproofs}.

\section{Discrepancy of hypergraphs and matrices} \label{sechy}

Let $X$ be a finite set, and let $\mathcal{E}$ be a family of subsets of $X$ (which are called \emph{(hyper)edges}). Then the pair $\Hy = (X,\ed)$ is called a (finite) \emph{hypergraph}. We can partition the elements of $X$ into two classes by a coloring function $b:~X \to \{-1,1\}$. The discrepancy of $\Hy$ is then defined as
$$
\disc(\Hy) = \min_{b:~X \to \{-1,1\}} ~\max_{E \in \ed}~ | b(E)|,
$$
where $b(E) = \sum_{x \in E} b(x)$. The concept of discrepancy of hypergraphs can be generalized to discrepancy of matrices in a natural way. Let $A = (a_{ij})_{1 \leq i \leq m,1 \leq j \leq n}$ be an $m \times n$-matrix, and set
$$
\disc(A) = \min_{b \in \{-1,1\}^n} \|A b\|_\infty. 
$$ 
This notion of discrepancy of matrices really contains the notion of discrepancy of hypergraphs; to see this, let $X=\{x_1, \dots, x_n\}$ and $\ed = \{E_1, \dots, E_m\}$ and set $A= (a_{ij})$, where $a_{ij}=1$ whenever $x_j \in E_i$ and $a_{ij}=0$ otherwise. Then $A$ is called the \emph{incidence matrix} of $\Hy$ and $disc(A)=disc(\Hy)$.\\

There are two closely related notions of discrepancy of matrices. One of them is the \emph{linear discrepancy}, which is defined as 
$$
\lindisc(A) = \max_{\beta \in [-1,1]^n} \min_{b \in \{-1,1\}^n} \| A (\beta-b) \|_\infty.
$$
The second is the \emph{hereditary discrepancy}, which is defined as 
$$
\herdisc(A) = \max_{J \subset \{1, \dots, n\}} \disc ((a_{ij})_{1 \leq i \leq m,j \in J}).
$$
Both notions can also be used to define the respective discrepancies for hypergraphs, by identifying a hypergraph with its incidence matrix. Then $\herdisc(\Hy)$ is the maximum discrepancy of all induced subgraphs of $\Hy$, while $\lindisc(\Hy)$ gives a bound for the error for approximating reals by integers (in the spirit of Lemmas~\ref{beckfiala} and~\ref{lemmahy} below).\\

An important relation which we will need is that for any $A$ we have
\begin{equation} \label{becks}
\lindisc(A) \leq 2 \herdisc(A);
\end{equation}
this inequality is due to Beck--Spencer~\cite{bsi} and Lov{\'a}sz--Spencer--Vesztergombi~\cite{lsv}. More information on the discrepancy of hypergraphs and matrices can be found in the books of Chazelle~\cite{cha} and Matou{\v{s}}ek~\cite{mat}.\\

Let $\Delta(\Hy)$ denote the \emph{maximum degree} of a hypergraph $\Hy$, that is 
$$
\Delta(\Hy) = \max_{x \in X} \# \{E \in \ed:~x \in E\}.
$$
In other words, no element $x$ of $X$ is contained in more than $\Delta(\Hy)$ sets $E \in \ed$. The Beck--Fiala theorem states that
\begin{equation} \label{beckf}
\disc(\Hy) \leq 2 \Delta(\Hy) -1
\end{equation}
for any hypergraph $\Hy$. Clearly the maximum degree of an induced subgraph of $\Hy$ cannot exceed the maximum degree of $\Hy$. Thus by a combination of~\eqref{becks} and~\eqref{beckf} we have
\begin{equation} \label{hye}
\lindisc(\Hy) \leq 4 \Delta(\Hy) -2
\end{equation}
for any hypergraph $\Hy$. A direct consequence of this inequality is the following lemma, which is (in a slightly stronger form) the key ingredient in Beck's proof of~\eqref{bec} in~\cite{beck}. 
\begin{lemma} \label{beckfiala}
Let real numbers $\beta_1, \dots, \beta_n \in [0,1]$ be given, and let $\ed$ be a family of subsets of $\{1, \dots, n\}$. Assume that each $i \in \{1, \dots, n\}$ belongs to at most $\Delta$ elements of $\ed$. Then there exist integers $b_1, \dots, b_n \in \{0,1\}$ such that
$$
\left| \sum_{i \in E} \beta_i - \sum_{i \in E} b_i \right| \leq 2\Delta -1 \qquad \textup{for all $E \in \ed$}.
$$
\end{lemma}
(Note that this lemma corresponds to using ``colors'' $\{0,1\}$ rather than $\{-1,1\}$, which saves us a factor 2 in comparison with~\eqref{hye}).\\

The Beck--Fiala theorem is remarkable insofar as the upper bound in~\eqref{beckf} depends only on $\Delta(\Hy)$, but \emph{not} on the number of vertices $n$ or the number of edges $m$. However, it turns out that we can improve~\eqref{bec} by using a different bound for the discrepancy of hypergraphs, which gives an improved dependence on $\Delta(\Hy)$ at the cost of an additional dependence on $m$. The following lemma is a consequence of a result of Banaszczyk~\cite{bana}; we will first state his result, then the lemma we need, and then show how the latter can be deduced from the former.\\

Banaszczyk proved the following (\cite[Theorem 1 and the remark on p.~353]{bana}): Let $K$ be a convex body in $\R^m$ whose $m$-dimensional (standard) Gaussian measure $\gamma_m$ is at least $\gamma_m(K) \geq 1/2$ and which contains the origin in its interior. Then for any points $u_1, \dots, u_n \in \R^m$ satisfying $\|u_i\|_2 \leq 1/5,~1 \leq i \leq n$, there exist signs $b_1, \dots, b_n \in \{-1,1\}$ such that $b_1 u_1 + \dots + b_n u_n \in K$.\\

\begin{lemma} \label{lemmahy}
Let real numbers $\beta_1, \dots, \beta_n \in [0,1]$ be given, and let $\ed$ be a family of subsets of $\{1, \dots, n\}$. Assume that each $i \in \{1, \dots, n\}$ belongs to at most $\Delta$ elements of $\ed$, and let $m$ denote the cardinality of $\ed$. Then there exist integers $b_1, \dots, b_n \in \{0,1\}$ such that
$$
\left| \sum_{i \in E} \beta_i - \sum_{i \in E} b_i \right| \leq 5 \sqrt{2 \Delta \log (2m)} \qquad \textup{for all $E \in \ed$}.
$$
Here we can choose $b_i=0$ whenever $\beta_i=0$.
\end{lemma}

The fact that Banaszczyk's theorem implies an upper bound for the discrepancy of a hypergraph of given maximum degree is known, and is for example mentioned at the end of Section 4.3 of Matou{\v{s}}ek's book~\cite[p.~115-116]{mat} (where the upper bound $\mathcal{O}(\sqrt{\Delta \log n})$ is given). However, we have not found any place where the value of the implied constant is explicitly stated; thus a detailed deduction of Lemma~\ref{lemmahy} is given below.\\

Banaszczyk's theorem can also be applied to another interesting problem. Let $u_1, \dots, u_n \in \R^m$ be vectors having Euclidean norm at most 1. Then by Banaszczyk's theorem there exist signs $b_1, \dots, b_n \in \{-1,1\}$ such that $\| b_1 u_1 + \dots + b_n u_n \|_\infty = \mathcal{O}(\sqrt{\log m})$. A famous conjecture of Komlos states that here the factor $\mathcal{O}(\sqrt{\log m})$ can be replaced by an absolute constant. If the Komlos conjecture is true, then the upper bound in~\eqref{beckf} can be replaced by $\mathcal{O}(\sqrt{\Delta(\Hy)})$ (which is known as the Beck--Fiala conjecture), the upper bound in Lemmas~\ref{beckfiala} and~\ref{lemmahy} can also be replaced by $\mathcal{O}(\sqrt{\Delta})$, and the asymptotic order of the logarithmic terms in Theorems~\ref{th1} and~\ref{th2} can be improved from $(3d+1)/2$ and $(3d+4)/2$ to $3d/2$ and $(3d+3)/2$, respectively.

\begin{proof}[Proof of Lemma~\ref{lemmahy}]
The density of the $m$-dimensional standard Gaussian measure is given by
$$
\gamma_n (y_1,  \dots, y_m) = (2 \pi)^{-m/2} \exp \left( - \frac{y_1^2 + \dots + y_m^2}{2} \right), \qquad (y_1, \dots, y_m) \in \R^m.
$$
Consequently for $K = [-\sqrt{2 \log (2m)},\sqrt{2 \log (2m)}]^m$ we have
\begin{eqnarray*}
\gamma_m(K) & = & \left(\frac{1}{\sqrt{2 \pi}} \int_{-\sqrt{2 \log (2m)}}^{\sqrt{2 \log (2m)}} e^{-y^2/2} ~dy \right)^m \\
& \geq & \left( 1 - \frac{2}{\sqrt{2 \pi}} \int_{\sqrt{2 \log (2m)}}^{\infty} \frac{y}{\sqrt{2 \log (2m)}} e^{-y^2/2} ~dy \right)^m  \\
& = & \left( 1 - \frac{1}{\sqrt{\pi \log (2m)}} e^{-\log (2m)}  \right)^m  \\
& = & \left( 1 - \frac{1}{2m \sqrt{\pi \log (2m)}} \right)^m \\
& \geq & 1/2
\end{eqnarray*}
(note that $\sqrt{\pi \log (2m)} \geq 1$ for $m \geq 1$). Let $X=\{1, \dots, n\}$ and the family $\ed$ of subsets of $X$ be given, and write $\Hy = (X,\ed)$. Then by the assumptions of the lemma the hypergraph $\Hy$ has maximum degree at most $\Delta$, and consequently the incidence matrix $A$ of $\Hy$ has column vectors $u_1, \dots, u_n$ for which $\|u_i\|_2 \leq \sqrt{\Delta},~1 \leq i \leq n$. Thus by Banaszczyk's theorem there exists colors $b_1, \dots, b_n \in \{-1,1\}$ such that
$$
\| b_1 u_1 +  \dots + b_n u_n \|_\infty \leq 5 \sqrt{2 \Delta \log (2m)}.
$$
In other words, for the discrepancy of $\Hy$ we have
$$
\disc (\Hy) \leq 5 \sqrt{2 \Delta \log (2m)}.
$$
Since the maximum degree of any induced subgraph of $\Hy$ is also at most $\Delta$, together with~\eqref{becks} this yields
$$
\lindisc(\Hy) \leq 10 \sqrt{2 \Delta \log (2m)},
$$
which implies Lemma~\ref{lemmahy} (note again that in Lemma~\ref{lemmahy} we use ``colors'' $\{0,1\}$ instead of $\{-1,1\}$, which saves us a factor 2). The fact that we can choose $b_i=0$ whenever $\beta_i=0$ is trivial (we can simply ignore all vertices $i$ for which $\beta_i=0$ in the application of Banaszczyk's theorem).
\end{proof}

\section{Empirical processes} \label{sec2}

One important ingredient in the proofs of Theorems~\ref{th1} and~\ref{th2} is the fact that any measure $\mu$ on $[0,1]^d$ can be approximated by a discrete measure. In this section it is shown how the existence of such an approximation, together with convergence rates, can be proved by considering the empirical process of $\mu$-distributed independent random variables, and how large deviation bounds for empirical processes indexed by sets can be used to obtain a generalization of the \emph{inverse of the discrepancy}-result~\eqref{hnwwb}. The results which we obtain in this section are much stronger than what would be necessary for the proofs of Theorems~\ref{th1} and~\ref{th2}, but they are of some independent interest.\\

Throughout this section, let $X_1, X_2, \dots$ be a sequence of independent, identically distributed random variables taking values in a measurable space $(X,\mathcal{A})$ and having distribution $P$, and let $\mathcal{C}$ be a class of subsets of $X$. For $C \in \mathcal{C}$ let $P_N(C)$ denote the empirical distribution of $X_1, \dots, X_N$, that is,
$$
P_N = \frac{1}{N} \sum_{n=1}^N \delta_{X_n},
$$
where $\delta_Y$ denotes the Dirac measure centered on $Y$, and let
$$
\alpha_N (C) = N^{1/2} (P_N(C) - P(C)), \qquad C \in \mathcal{C},
$$
denote the empirical process indexed by $\mathcal{C}$.  To avoid measurability problems we will throughout this paper assume that $\mathcal{C}$ is countable. In this section we will be concerned with probabilistic estimates for the size of 
\begin{equation}
\sup_{C \in \mathcal{C}} |\alpha_N(C)|,
\end{equation}
which of course will depend on the complexity of the class $\mathcal{C}$. Let $A$ be a finite subset of $X$. Then $\mathcal{C}$ is said to shatter $A$ if for every subset $B$ of $A$ there exists a $C \in \mathcal{C}$ such that $B = A \cap C$. If there exists a largest finite number $d$ such that $\mathcal{C}$ shatters at least one set of cardinality $d$, then $\mathcal{C}$ is called a \emph{Vapnik--\v{C}ervonenkis class} (VC-class) of index $d$. The notion of VC-classes is well-established in the theory of empirical processes indexed by sets, since there exist strong bounds for the metric entropy of such classes. Using such entropy bounds (due to Haussler~\cite{hau}), Talagrand~\cite[Theorem 1.1 (i)]{tala} proved that there exists an absolute constant $c_{\textup{abs}}$ such that for each VC-class $\mathcal{C}$ of index $d \geq 1$ we have for all $t>0$
\begin{equation} \label{tala}
\p \left( \sup_{C \in \mathcal{C}} |\alpha_N(C)| \geq t\right) \leq \frac{c_{\textup{abs}}}{t} \left( \frac{c_{\textup{abs}} t^2}{d}\right)^d e^{-2t^2}.
\end{equation}

Talagrand's result, which is the key ingredient in the proof of~\eqref{hnwwb}, is a significant improvement of earlier results of a similar type (for a comparison with earlier results, see Table 1 in~\cite{vad}). Regrettably Talagrand doesn't give an explicit value for the value of the constant $c_\textup{abs}$ in~\eqref{tala}. It seems that in principle Talagrand's method would allow to obtain an explicit estimate for $c_{\textup{abs}}$; however, he writes: ``\emph{We have, however, felt that the search of sharp numerical constants is better left to others with the talent and the taste for it}'' (\cite[p. 31]{tala}), and apparently no one has carried out these calculations since then. Yet there do exist weaker versions of~\eqref{tala} which do not involve any unknown constants. One result of this type is the original estimate of Vapnik and \v{C}ervonenkis~\cite{vc}, which together with Sauer's lemma~\cite{sau} gives 
\begin{equation} \label{vc}
\p \left( \sup_{C \in \mathcal{C}} |\alpha_N(C)| \geq t \right) \leq 8 \left(\frac{eN}{d}\right)^d e^{-t^2/32}
\end{equation}
for $t > 0$ (for a textbook treatment, see for example~\cite[Theorem 12.5]{dgla}). Note that~\eqref{vc} is weaker than~\eqref{tala} in both its dependence on $N$ and $t$, and could not be used to prove~\eqref{hnwwb}. Another completely explicit large deviations bound for empirical processes, which will be perfectly suitable for our purpose, is due to Alexander. We state it as a lemma.

\begin{lemma} \label{lemmaalex}
Let $N \geq 1$ and let $\mathcal{C}$ be a countable VC-class of index $d$. Then for any $t$ satisfying both
\begin{equation} \label{alexcond1}
t > \frac{2^{33/2} d}{\sqrt{N}} \log\left(\max\left(\frac{N}{2d},e\right)\right)
\end{equation}
and
\begin{equation} \label{alexcond2}
t > \sqrt{2^{25} d \log(4)}.
\end{equation}
we have
$$
\p \left( \sup_{C \in \mathcal{C}} |\alpha_N(C)| > t \right) \leq 16 e^{-t^2}.
$$
\end{lemma}

\emph{Remarks:~} This is a special case of~\cite[Theorem 2.8 (ii)]{alex}. In fact, Alexander's theorem is much more general than Lemma~\ref{lemmaalex}, insofar as it is formulated for a VC graph class of functions with $(d,k)$-constructible graph region class rather than for a VC class of sets. However, as noted on~\cite[p. 1049]{alex}, a VC class of index $d$ is always a VC graph class with $(d,1)$-constructible graph region class; moreover, every result formulated for a class of functions is clearly also applicable for a class of sets (by considering the indicator functions of the sets). Furthermore, for the parameters in Alexander's theorem we have chosen $\psi=2M^2$ (which is permitted by the remark after~\cite[Equation (1.7)]{alex}) and $\ve=1/2$. Since our functions $f$ are of the form $\mathds{1}_{C}(X_n)$ for some sets $C$ we have $0 \leq f(X_n) \leq 1$, and hence the variance of $f(X_n)$ is at most $1/4$; consequently, we can choose $\alpha = 1/4$. Finally, the estimate in Alexander's theorem is for 
the outer measure $\p^*$ rather than the measure $\p$; however, by assuming that $\mathcal{C}$ is countable we can avoid all measurability problems, which means that the set $\{ \sup |\alpha_N| > t \}$ is measurable and we can use the measure $\p$.\\

As mentioned before, Talagrand's inequality~\eqref{tala} is the key ingredient in the proof of~\eqref{hnwwb}, together with the well-known fact that the class of axis-parallel boxes in $[0,1]^d$ which have one vertex in the origin is a VC-class of index $d$. However, as observed by Heinrich \emph{et al.}, Talagrand's inequality~\eqref{tala} can be used to obtain a much more general \emph{inverse of the star-discrepancy} result for VC-classes. Using the notation from above, let $t_1, \dots, t_N$ be points in $X$ and set
\begin{equation*} \label{discc}
D_N^{(\mathcal{C},P)} (t_1, \dots, t_N) = \sup_{C \in \mathcal{C}} \left| \frac{1}{N} \sum_{n=1}^N \mathds{1}_C(t_n) - P(C) \right|.
\end{equation*}
Assume that $\mathcal{C}$ is a VC-class of index $d$. Then by~\eqref{tala} there exist points $t_1, \dots, t_N \in X$ such that
\begin{equation} \label{hnww2}
D_N^{(\mathcal{C},P)} (t_1, \dots, t_N) \leq c_{\textup{abs}} \frac{\sqrt{d}}{\sqrt{N}},
\end{equation}
which means that there always exists a point set of cardinality at most $c_{\textup{abs}} d \ve^{-2}$ achieving a $D_N^{(\mathcal{C},P)}$-discrepancy of at most $\ve$. Since Talagrand's inequality contains an unspecified constant, the same is the case for~\eqref{hnww2}. However, as we will see in the sequel, using Lemma~\ref{lemmaalex} instead of~\eqref{tala} we can obtain a fully explicit version of~\eqref{hnww2}, which we state as the following theorem.

\begin{theorem} \label{th3}
Let $(X,\mathcal{A},P)$ be a probability space, and let $\mathcal{C}$ be a class of subsets of $X$. Assume furthermore that $\mathcal{C}$ is a countable VC-class of index $d$. Then for any $N \geq 1$ there exist points $t_1, \dots, t_N \in X$ for which
$$
D_N^{(\mathcal{C},P)} (t_1, \dots, t_N)  \leq 2^{13} \frac{\sqrt{d}}{\sqrt{N}}.
$$
Consequently, for any $\ve \in (0,1)$ there exists a set of a most $N(\ve)=2^{26} d \ve^{-2}$ points whose $D_N^{(\mathcal{C},P)}$-discrepancy does not exceed $\ve$.
\end{theorem}

Theorem~\ref{th3} is the first fully explicit \emph{inverse of the discrepancy}-type result for VC classes. However, the value of the constants in Theorem~\ref{th3} is not very satisfactory, given the fact that in inequality~\eqref{hnwwb} (that is, in the case of the star-discrepancy on $[0,1]^d$) we may choose $c_{\textup{abs}} = 100$ (see~\cite{aist}). One way for a possible improvement of the constants in Theorem~\ref{th3} would of course be to try to find a explicit version of Talagrand's inequality~\eqref{tala}. Another possible way would be to use Massart's~\cite{mass} explicit version of an other large deviations bounds for empirical processes due to Talagrand~\cite{tala2}; however, this bound involves the quantity 
\begin{equation} \label{exp}
\E \left( \sup_{C \in \mathcal{C}} |\alpha_N(C)| \right),
\end{equation}
which depends on the entropy of the class $\mathcal{C}$. Using Haussler's entropy bounds for VC classes and the method for estimating~\eqref{exp} which is indicated in~\cite[Section 5.2]{gna} it might be possible to obtain an improved version of Theorem~\ref{th3}. For the problem of estimating~\eqref{exp} in the context of low-discrepancy point sets on $[0,1]^d$ see~\cite{aistho, doerr}.\\

To conclude this section, we repeat an argument from~\cite{hnww} to show how to apply the results obtained for the discrepancy $D_N^{(\mathcal{C},P)}$ in Theorem~\ref{th3} to the case of the discrepancy $D_N^*(~\cdot~ ; \mu)$, as defined in~\eqref{discmu}. Clearly, using the notation from the present section, we have to choose $X = [0,1]^d$ and $P = \mu$. However, we cannot choose $\mathcal{C}$ as the class of all axis-parallel boxes in $[0,1]^d$ having one vertex at the origin, as this class obviously is not countable. Instead we choose $\mathcal{C}$ as the class of all such boxes which have only rational coordinates. Restricting the test sets $A$ in definition~\eqref{discmu} to boxes having only rational coordinates does not change the value of the star-discrepancy. Consequently, with these settings the discrepancies $D_N^{(\mathcal{C},P)}$ and $D_N^*(~\cdot~; \mu)$ coincide, and Theorem~\ref{th3} includes an \emph{inverse of the star-discrepancy} result for the discrepancy $D_N^*(~\cdot~; \mu)$, which is 
stated as a corollary below.

\begin{corollary} \label{co1}
Let $\mu$ be any non-negative, normalized Borel measure on $[0,1]^d$. Then for any $N \geq 1$ there exist points $x_1, \dots, x_N \in [0,1]^d$ such that
$$
D_N^*(x_1, \dots, x_N;\mu) \leq 2^{13} \frac{\sqrt{d}}{\sqrt{N}}.
$$
Furthermore, for any $\ve \in (0,1)$ there exists a set of a most $N(\ve)=2^{26} d \ve^{-2}$ points whose star-discrepancy with respect to $\mu$ does not exceed $\ve$.
\end{corollary}

\section{Numerical integration of discontinuous functions} \label{sec3}

The Koksma--Hlawka inequality~\eqref{kokshla} can be rewritten for smooth functions $f$ on $[0,1]^d$ in the form
\begin{equation} \label{kh2}
\left| \frac{1}{N} \sum_{n=1}^N f(x_n) - \int_{[0,1]^d} f(x)~dx \right| \leq D_N^*(x_1, \dots, x_N) \cdot V_{1,1}(f),
\end{equation}
where
\begin{equation} \label{varnormal}
V_{1,1}(f) = \sum_{\mathfrak{u} \subset \{1, \dots, d\}} \int_{[0,1]^{|\mathfrak{u}|}} \left| \frac{\partial^{|\mathfrak{u}|}f}{\partial x_\mathfrak{u}}(x_\mathfrak{u};1) \right| dx_{\mathfrak{u}}.
\end{equation}
Here the sum is extended over all subsets $\mathfrak{u}$ of $\{1, \dots, d\}$. The symbol $x_\mathfrak{u}$ means that only those coordinates of $x$ are considered whose index is contained in $\mathfrak{u}$, and $(x_\mathfrak{u};1)$ means that all those components of $x$ whose index is not contained in $\mathfrak{u}$ are replaced by 1. The quantity $V_{1,1}(f)$ is the norm corresponding to an inner product in a certain reproducing kernel Hilbert space of Sobolev type; for more details, see for example~\cite[Chapter 3]{dksh} or~\cite[Chapter 2]{dpd}.\\

The inequality~\eqref{kh2} allows a nice function-analytic interpretation of the error of Quasi-Monte Carlo integration for smooth functions; however, it is not of any use if the function $f$ has discontinuities. The Koksma--Hlawka inequality in the form~\eqref{kokshla} is slightly more flexible, since it also yields upper bounds for the numerical integration of discontinuous functions. In particular, this is the case if $f$ is the indicator function of an axis-parallel box having one vertex at the origin (which is not surprising in consideration of the definition of the star-discrepancy). However, the variation of a function $f$ in the sense of Hardy and Krause can be infinity even if $f$ is a rather simple function, for example the indicator function of a polytope which does not only have axis-parallel faces.\\

Let $\Omega$ be a subset of $[0,1]^d$, and assume that $f(x)$ is of the form 
\begin{equation} \label{omegacl}
g(x) \cdot \mathds{1}_\Omega(x),
\end{equation}
where $g$ is a smooth function. Let $x_1, \dots, x_N$ be points in $[0,1]^d$. Can we then obtain a bound for the QMC integration error
\begin{equation} \label{err}
\left| \frac{1}{N} \sum_{n=1}^N f(x_n) - \int_{[0,1]^d} f(x)~dx \right|
\end{equation}
in terms of an (appropriately defined) discrepancy of $x_1, \dots, x_N$, which holds \emph{uniformly} for all possible sets $\Omega$? Probably not without strong a-priori restrictions on the class of sets in which $\Omega$ has to lie. For example, even if we set $g \equiv 1$ but allow $\Omega$ to be any subset of $[0,1]^d$, there will always be a choice of $\Omega$ such that the error in~\eqref{err} equals 1 (namely if we take $\Omega = [0,1]^d \backslash \{x_1, \dots, x_N\}$). Also if we impose some restrictions on $\Omega$ we will not necessarily get any good bounds for the QMC error~\eqref{err}: For example, if $g \equiv 1$ and $\Omega$ can be any convex set, then independent of the point set $x_1, \dots, x_N$ the error in~\eqref{err} will not be better than $\gg N^{-2/(d+1)}$ for some choice of $\Omega$ (a variant of the Koksma--Hlawka inequality for functions having convex super-level sets has actually recently been proved by Harman~\cite{harm}). In other words, we cannot expect to find a point set 
which gives good error bounds for QMC integration for functions of the form~\eqref{omegacl} which hold \emph{uniformly} for all sets $\Omega$ from a large class of subsets of $[0,1]^d$.\\

However, the problem of numerical integration of functions having discontinuities cannot be entirely neglected, since such functions arise in a natural way in one of the main areas of applications of QMC methods. More precisely, QMC methods are a standard tool for pricing financial derivatives in financial mathematics, and there are classes of such derivatives, such as digital options and barrier options, which have a discontinuous payoff function. In a standard model of financial mathematics the value of the underlying asset is assumed to follow a geometric Brownian motion, and the value of a financial derivative is given by the expected value of the payoff under the risk-neutral measure. The Brownian motion is discretized into $d$ intervals, and points in $[0,1]^d$ are used, together with a suitable path-generating method, to simulate a path of the geometric Brownian motion. Accordingly, the payoff of the financial derivative can be estimated by the $d$-dimensional integral of a function $f$, where the 
properties of $f$ depend on both the payoff of the derivative and on the path-generating method which is used. For the basic principles of this method see Glasserman~\cite{glass}. The influence of the path-generating method on the effectiveness of QMC integration has been studied in detail by Wang and Tan~\cite{wangtan1,wangtan2}.\\

In the sequel we will show that once we assume that the set $\Omega$ (where the discontinuities occur) is \emph{fixed}, then we can find QMC point sets with asymptotic convergence rate almost as good as~\eqref{lowd}. The key ingredient (besides Theorem~\ref{th1}) is a recent variant of the Koksma--Hlawka inequality for general integration domains, which is due to Brandolini, Colzani, Gigante and Travaglini~\cite{bcgt}. They proved (amongst many other results) the following inequality, which we state as a lemma.

\begin{lemma} \label{lemmabcgt}
Let $f$ be a continuous function on $\mathbb{T}^d$, and let $\Omega$ be a Borel subset of $[0,1]^d$. Then for $x_1, \dots, x_N$ in $[0,1]^d$ we have
$$
\left| \frac{1}{N} \sum_{n=1}^N (f \cdot \mathds{1}_\Omega) (x_n) - \int_\Omega f(x)~dx \right| \leq D_N^{\Omega}(x_1, \dots, x_N) \cdot V(f),
$$
where
\begin{equation} \label{discomega}
D_N^{(\Omega)} (x_1, \dots, x_N) = 2^d \sup_{A \subset [0,1]^d} \left| \frac{1}{N} \sum_{n=1}^N \mathds{1}_{\Omega \cap A} (x_n) - \lambda(\Omega \cap A) \right|
\end{equation}
and
\begin{equation} \label{bcgtvar}
V(f) = \sum_{\mathfrak{u} \subset \{1, \dots, d\}} 2^{d-|\mathfrak{u}|} \int_{[0,1]^d} \left| \frac{\partial^{|\mathfrak{u}|}}{\partial x_\mathfrak{u}} f(x) \right| dx,
\end{equation}
where $\frac{\partial^{|\mathfrak{u}|}}{\partial x_\mathfrak{u}} f(x)$ is the partial derivative of $f$ with respect to those components of $x$ with index in $\mathfrak{u}$, and the supremum is taken over all axis-parallel boxes $A$.
\end{lemma} 

As a consequence we get the following theorem.

\begin{theorem} \label{th4}
Let $f$ by a continuous function on $\mathbb{T}^d$, and let $\Omega$ be a Borel subset of $[0,1]^d$ which has positive $d$-dimensional Lebesgue measure $\lambda$. Then for any $N \geq 1$ there exist points $x_1, \dots, x_N$ in $[0,1]^d$ for which
\begin{equation} \label{omeg}
\left| \frac{1}{N} \sum_{n=1}^N f (x_n) - \frac{1}{\lambda(\Omega)} \int_\Omega f(x)~d\lambda(x) \right| \leq \frac{63 \sqrt{d} ~4^d \left(2 + \log_2 N \right)^{(3d+1)/2} + 1}{N} + \frac{\|f\|_\infty}{N}
\end{equation}
where $V(f)$ is defined in~\eqref{bcgtvar}.
\end{theorem}

The purpose of Theorem~\ref{th4} is to show that \emph{in principle} there exist QMC point sets which yield fast asymptotic convergence rates for QMC integration for functions of the form~\eqref{omegacl}. Note, however, that the point set $x_1, \dots, x_N$ in Theorem~\ref{th4} depends on $\Omega$, and consequently we would have to construct \emph{different} point sets for functions having their discontinuities at different positions (which is not surprising, as noted at the beginning of this section). Thus for calculating the risk-neutral price of a financial derivative with discontinuous payoff the point set yielding small QMC integration error has to be chosen with respect to both the barrier (where the discontinuity occurs) and the path-generating method (but \emph{not} depending on the value of the payoff). Additionally equation~\eqref{omeg} contains the term $\lambda(\Omega)$, which might be unknown. Furthermore the proof of Theorem~\ref{th1} is of a purely probabilistic/combinatorial nature and \emph{
not} constructive, so it remains an open problem how to construct such point sets (possibly in a parametric form, depending on $\Omega$). Thus Theorem~\ref{th4} is far away from a possible implementation, and is mostly of theoretical interest.\\

Note that the assumption $\lambda(\Omega) >0$ in Theorem~\ref{th4} is no real restriction, since in the case $\lambda(\Omega)=0$ we clearly have
$$
\int_\Omega f(x)~d\lambda(x)=0
$$
for any function $f(x)$, which makes numerical integration pointless. Note also that Theorem~\ref{th4} is \emph{not} a trivial consequence of Theorem~\ref{th1}, combined with Lemma~\ref{lemmabcgt}, since the two notions of discrepancy in~\eqref{discmu} and~\eqref{discomega} are defined in a different way. We will choose $\mu$ in such a way that
\begin{equation} \label{mua}
\mu(A) = \frac{\lambda(\Omega \cap A)}{\lambda(\Omega)}. 
\end{equation}
Then $\mu$ is a normalized measure. However, then the discrepancy defined in~\eqref{discmu} is given by
\begin{equation} \label{discmumod}
\sup_{A \subset [0,1]^d} \left| \frac{1}{N} \sum_{n=1}^N \mathds{1}_A (x_n) - \frac{\lambda(\Omega \cap A)}{\lambda(\Omega)} \right|,
\end{equation}
which is not the same as~\eqref{discomega}; in particular, these discrepancies are only comparable for our point set from Theorem~\ref{th1} if we can guarantee that all points constructed there for the measure $\mu$ given by~\eqref{mua} are contained in $\Omega$ (which actually is the case, as the proof of Theorem~\ref{th4} will show).

\section{Open problems} \label{secop}

As mentioned in~\eqref{lowd} for any $d \geq 1$ and $N \geq 1$ there exists a (finite) point set $x_1, \dots, x_N$ whose star-discrepancy (with respect to the uniform measure on $[0,1]^d$) is bounded by
\begin{equation} \label{disc1}
D_N^*(x_1, \dots, x_N) \ll (\log N)^{d-1} N^{-1};
\end{equation}
here the implied constant depends on $d$. Furthermore, there exists an infinite sequence $(x_n)_{n\geq 1}$ of points from $[0,1]^d$ whose star-discrepancy is bounded by
\begin{equation} \label{disc2}
D_N^*(x_1, \dots, x_N) \ll (\log N)^{d} N^{-1} \qquad \textrm{for all $N \geq 1$},
\end{equation}
where again the implied constant depends on $d$. It is unknown if these upper bounds are optimal, and the problem of finding the optimal convergence rates in these inequalities is called the \emph{Great Open Problem} of discrepancy theory. The best known lower bound is
$$
D_N^*(x_1, \dots, x_N) \gg (\log N)^{(d-1)/2 + \eta(d)}
$$
for some $\eta>0$, and for any (finite) set of $N$ points in $[0,1]^d$, due to Bilyk, Lacey and Vagharshakyan~\cite{blv} (see~\cite{bor} for a survey). Obviously the upper bounds in Theorem~\ref{th1} and~\ref{th2} are of asymptotic order
\begin{equation} \label{disc3}
D_N^*(x_1, \dots, x_N;\mu) \ll (\log N)^{(3d+1)/2} N^{-1}
\end{equation}
for a finite point set $x_1, \dots, x_N$, and 
\begin{equation} \label{disc4}
D_N^*(x_1, \dots, x_N;\mu) \ll (\log N)^{(3d+4)/2} N^{-1}  \qquad \textrm{for all $N \geq 1$}
\end{equation}
for an infinite sequence $(x_n)_{n \geq 1}$; it would be very interesting to know if the asymptotic upper bounds in these inequalities can be improved to~\eqref{disc1} and~\eqref{disc2}, respectively. As mentioned in Section~\ref{sechy}, a proof of the Komlos conjecture (or the Beck--Fiala conjecture) would imply that the exponents of the logarithmic terms in~\eqref{disc3} and~\eqref{disc4} could both be reduced by $1/2$, which would still leave a gap to~\eqref{disc1} and~\eqref{disc2}.\\

As mentioned in Section~\ref{sec1}, the problem concerning the inverse of the discrepancy is open even in the case of the ``classical'' star-discrepancy (that is, with respect to the uniform measure). As noted in Section~\ref{sec2} the upper bound~\eqref{hnwwb} remains valid if the uniform measure is replaced by an other measure on $[0,1]^d$; on the other hand, Hinrichs' proof of the lower bound~\eqref{hin} depends on the metric entropy of the system of axis-parallel boxes \emph{with respect to the uniform measure}, and can (as far as we understand) not easily be generalized to other measures. Thus many problems concerning the inverse of the discrepancy for the star-discrepancy $D_N^*$, the discrepancy $D_N^*(\cdot,\mu)$ and the discrepancy $D_N^{(\mathcal{C},P)}$ remain open. As noted after the statement of Theorem~\ref{th3}, it would also be interesting to get an improved version of this theorem with a significantly smaller value for the numerical constant, in the general case of the inverse of the 
discrepancy with respect to general measures and VC-classes.\\

As far as we know, Lemma~\ref{lemmabcgt} of Brandolini, Colzani, Gigante and Travaglini is the first Erd\H os-Tur\'an inequality of this general type (together with an inequality of Harman~\cite{harm}, which seems to be less useful for numerical integration). Generalizing and improving this inequality would be very interesting. In particular it would be helpful to get rid of the terms which are exponential in $d$, and which spoil a possible application of the point set from Corollary~\ref{co1} in order to obtain a version of  Theorem~\ref{th4} which is suitable for point sets which have moderate cardinality $N$ in comparison with $d$. Furthermore it would be nice to have a version of Lemma \ref{lemmabcgt} which is formulated for the discrepancy in \eqref{discmumod} rather than \eqref{discomega}; such a Koksma-Hlawka inequality would be more convenient in the case when $\Omega$ is fixed and has small volume, and could for example render the laborious technical argument in the proof of Theorem~\ref{th4} 
unnecessary.

%Furthermore, with an application of Theorem~\ref{th1} in mind it would be more natural to prove a version of this generalized Erd\H os-Tur\'an inequality for the discrepancy defined by
%$$
%\sup_{A \subset [0,1]^d} \left| \frac{1}{N} \sum_{n=1}^N \mathds{1}_{\Omega \cap A} (x_n) - \frac{\lambda(\Omega \cap A)}{\lambda(A)} \right|
%$$
%rather than~\eqref{discomega}; such an inequality would reflect that for integrating over a smaller domain than $[0,1]^d$ (for example for $\Omega=[0,1/2]^d$) one should expect faster convergence than when integrating over the full unit cube.

\section{Proofs} \label{secproofs}
The proofs given here for Theorem~\ref{th1} and Theorem~\ref{th2} are similar to those given for~\eqref{bec} by Beck in~\cite{beck}, the only major differences being that we use Lemma~\ref{lemmahy} instead of Lemma~\ref{beckfiala}, and carefully take into account the dependence of all involved quantities on the dimension $d$.\\

As a consequence of Lemma~\ref{lemmahy}, we get the following lemma.
\begin{lemma} \label{lemmabeta}
For $(i_1, \dots, i_d) \in \{1, \dots, N\}^d$, let numbers $\beta_{(i_1, \dots, i_d)}$ from $[0,1]$ be given. Then there exist numbers $b_{(i_1, \dots, i_d)} \in \{0,1\}$ such that
$b_{(i_1, \dots, i_d)}=0$ whenever $\beta_{(i_1, \dots, i_d)}=0$, and such that for all $(J_1, \dots, J_d) \in \{1, \dots, N\}^d$ we have
$$
\left| \sum_{i_1=1}^{J_1} \dots \sum_{i_d=1}^{J_d} \left(b_{(i_1, \dots, i_d)} - \beta_{(i_1, \dots, i_d)} \right) \right| \leq 10 \sqrt{d} \left(2 + \log_2 N \right)^{(3d+1)/2}.
$$
\end{lemma}

\begin{proof}
For simplicity, we assume that $N = 2^m$ for some integer $m \geq 1$ (this assumption will be dropped at the end of the proof). We use a classical dyadic decomposition method. Let $m_1, \dots, m_d$ be non-negative integers such that $0 \leq m_1, \dots, m_d \leq m$, and let $\mathcal{E}(m_1, \dots, m_d)$ denote the class of all sets of the form
$$
\prod_{s=1}^d \{j_s 2^{m_s} +1, \dots, (j_s+1)2^{m_s} \},
$$ 
(where the product $\prod$ denotes the Cartesian product of sets), for  $0 \leq j_s < 2^{m-m_s}, ~1 \leq s \leq d$. Then for any fixed $m_1, \dots, m_d$ the class $\mathcal{E}(m_1, \dots, m_d)$ forms a partition of $\{1, \dots, N\}^d$, which implies that for fixed $m_1, \dots, m_d$ any $(i_1, \dots, i_d)$ appears in only one set of $\mathcal{E}(m_1, \dots, m_d)$. On the other hand, for fixed $m_1, \dots, m_d$ the class $\mathcal{E}(m_1, \dots, m_d)$ consists of at most $N^d$ elements. The number of possible values for $(m_1, \dots, m_d)$ is $\left(m + 1\right)^d$. Thus by Lemma~\ref{lemmahy} there exist numbers $b_{(i_1, \dots, i_d)} \in \{0,1\}$ such that for any set $E$ which is contained in $\mathcal{E}(m_1, \dots, m_d)$ for some $m_1, \dots, m_d$ we have
\begin{eqnarray*}
\left| \sum_{(i_1,\dots,i_d) \in E} \left( b_{(i_1, \dots, i_d)} - \beta_{(i_1, \dots, i_d)} \right) \right| & \leq & 5 \sqrt{2 (m+1)^d \log (2(m+1)^d N^d)} \\
& \leq & 10 \sqrt{d (m+1)^d \log (2N)},
\end{eqnarray*}
and such that additionally $b_{(i_1, \dots, i_d)}=0$ whenever $\beta_{(i_1, \dots, i_d)}=0$.\\

For arbitrary $J_1, \dots, J_d$ from $\{1, \dots, N\}$ the set
$$
\prod_{s=1}^d \{1, \dots, J_s\}
$$
can be written as a disjoint union of sets from 
$$
\bigcup_{0 \leq m_s \leq m ~\textrm{for $1\leq s \leq d$}} \mathcal{E}(m_1, \dots, m_d),
$$
in such a way that we take at most one set from $\mathcal{E}(m_1, \dots, m_d)$ for any possible $m_1, \dots, m_d$ (such a representation is easily found from the binary representation of the numbers $J_1, \dots, J_d$). Consequently we have
\begin{equation} \label{equ1}
\left| \sum_{(i_1,\dots,i_d) \in \prod_{s=1}^d \{1, \dots, J_s\}} \left( b_{(i_1, \dots, i_d)} - \beta_{(i_1, \dots, i_d)} \right) \right| \leq 10 \sqrt{d (m+1)^d \log (2N)} (m+1)^d
\end{equation}
for any $J_1, \dots, J_d$. Finally, assume that $N$ is not an integral power of $2$. Then we can set $\hat{N}$ for the smallest number exceeding $N$ which is an integral power of 2, and set
$$
b_{(i_1, \dots, i_d)} = 0 \quad \textrm{and} \quad \beta_{(i_1, \dots, i_d)} = 0 \qquad \textrm{for $(i_1, \dots, i_d) \in \{1, \dots, \hat{N} \}^d \big\backslash \{1, \dots, N \}^d$}.
$$
Then from~\eqref{equ1} we get 
\begin{eqnarray*}
& & \max_{\substack{J_1, \dots, J_d:\\1 \leq J_s \leq N~\textrm{for}~1 \leq s \leq d}}~ \left| \sum_{(i_1,\dots,i_d) \in \prod_{s=1}^d \{1, \dots, J_s\}} \left( b_{(i_1, \dots, i_d)} - \beta_{(i_1, \dots, i_d)} \right) \right| \\
& \leq & 10 \sqrt{d (\lceil \log_2 N \rceil+1)^d \log (4N)} (\lceil \log_2 N \rceil+1)^d \\
& \leq & 10 \sqrt{d} \left(2 + \log_2 N \right)^{(3d+1)/2},
\end{eqnarray*}
which proves Lemma~\ref{lemmabeta}.
\end{proof}

\begin{lemma} \label{lemma2} \label{lemmath1}
Given any (not necessarily distinct) points $z_1, \dots, z_K$ in $[0,1]^d$ and any number $N$ satisfying $N \leq \sqrt{K}$ we can find an $N$-element subset $\{x_1, \dots, x_N\}$ of $\{z_1, \dots, z_K\}$ such that for any axis-parallel box $A \subset [0,1]^d$ which has one vertex at the origin we have
$$
\left| \sum_{n=1}^N \mathds{1}_A (x_n) - \frac{N}{K} \sum_{n=1}^K \mathds{1}_A (z_n) \right| \leq 60 \sqrt{d} \left(2 + \log_2 N \right)^{(3d+1)/2} + 4d + 6.
$$
\end{lemma}
(Remark: Since we do not assume that the points $z_1, \dots, z_K$ are distinct, the sets in the statement and in the proof of this lemma are strictly speaking multisets. However, to keep the presentation short, we simply neglect this minor technicality.)

\begin{proof}
We can rearrange the points $z_1, \dots, z_K$ in $d$ different ways in increasing order, according to their $d$ coordinates. More precisely, for each $s \in \{1, \dots, d\}$ we define $z^{(s)}_1, \dots, z^{(s)}_K$ in such a way that
$$
\left\{ z^{(s)}_1, \dots, z^{(s)}_K \right\} = \{z_1, \dots, z_K\}
$$
and such that the $s$-th coordinate of $z^{(s)}_m$ does not exceed the $s$-th coordinate of $z^{(s)}_n$, for $m \leq n$.\\

For $1 \leq s \leq d$ and $1 \leq i \leq N$ we set
$$
F_i^{(s)} = \left\{ z^{(s)}_n:~\left\lfloor (i-1) \frac{K}{N} \right\rfloor < n \leq \left\lfloor i \frac{K}{N} \right\rfloor \right\},
$$
and, for $(i_1, \dots, i_d) \in \{1, \dots, N\}^d$,
$$
\beta_{(i_1, \dots, i_d)} = \frac{N}{K+N} ~ \# \left\{n:~z_n \in \bigcap_{s=1}^d F_{i_s}^{(s)} \right\}.
$$
By Lemma~\ref{lemmabeta} there exist numbers $b_{(i_1, \dots, i_d)} \in \{0,1\}$ such that for all $(J_1, \dots, J_d) \in \{1,\dots,N\}^d$ we have
$$
\left| \sum_{i_1=1}^{J_1} \dots \sum_{i_d=1}^{J_d} \left(b_{(i_1, \dots, i_d)} - \beta_{(i_1, \dots, i_d)} \right) \right| \leq 10 \sqrt{d} \left(2 + \log_2 N \right)^{(3d+1)/2},
$$
and such that $b_{(i_1, \dots, i_d)}=0$ whenever $\beta_{(i_1, \dots, i_d)}=0$. Now we select one point $x_{(i_1, \dots, i_d)}$ from $\prod_{s=1}^d F_{i_s}^{(s)}$ whenever $b_{(i_1, \dots, i_d)}=1$, and write $\mathcal{P}$ for the set consisting of all such points.\\

Let $J_1, \dots, J_d \in \{1, \dots, N\}$ be fixed, and set 
$$
G=G_{(J_1, \dots, J_d)}=\bigcup_{1 \leq i_1 \leq J_1} \dots \bigcup_{1 \leq i_d \leq J_d} \left( \bigcap_{s=1}^d F_{i_s}^{(s)} \right).
$$
Then we have
\begin{eqnarray}
& & \left| \# (\mathcal{P} \cap G) - \frac{N}{K} ~\# (G) \right| \nonumber\\
& \leq & \left| \# (\mathcal{P} \cap G) - \frac{N}{K+N} ~\# (G) \right| + \underbrace{\left| \left( \frac{N}{K} - \frac{N}{K+N} \right) ~\# (G) \right|}_{\leq N^2/(K+N) \leq 1} \nonumber\\
& \leq & \left| \sum_{i_1=1}^{J_1} \dots \sum_{i_d=1}^{J_d} \left(b_{(i_1, \dots, i_d)} - \beta_{(i_1, \dots, i_d)} \right) \right| + 1 \nonumber\\
& \leq & 10 \sqrt{d} \left(2 + \log_2 N \right)^{(3d+1)/2} + 1. \label{equ2a}
\end{eqnarray}

In particular, by choosing $J_1 = \dots = J_d = N$, we see that the cardinality of $\mathcal{P}$ differs from $N$ by at most $10 \sqrt{d} \left(2 + \log_2 N \right)^{(3d+1)/2} + 1$. Consequently, there exists an $N$-element subset $\mathcal{Q}$ of $\{z_1, \dots, z_K\}$ such that 
$$
|\#(\mathcal{P}) - \#(\mathcal{Q})| \leq 10 \sqrt{d} \left(2 + \log_2 N \right)^{(3d+1)/2} + 1,
$$ 
and by~\eqref{equ2a} we have
\begin{equation}\label{(5)}
\left| \# (\mathcal{Q} \cap G) - \frac{N}{K} ~\# (G) \right| \leq 20 \sqrt{d} \left(2 + \log_2 N \right)^{(3d+1)/2} + 2
\end{equation}
for all $G=G(J_1, \dots, J_d)$. In the sequel we will show that the point set $\mathcal{Q}$ satisfies the conclusion of Lemma~\ref{lemmath1}.\\

Let $(a_1, \dots, a_d) \in [0,1]^d$ be given, and let $A$ denote the axis-parallel box $\prod_{s=1}^d [0,a_s]$. Then there exist $J_1, \dots, J_d \in \{1, \dots, N-1\}$ such that
$$
G(J_1, \dots, J_d) \subset \{z_n:~z_n \in A\} \subset G(J_1+1, \dots, J_d+1).
$$
Furthermore, since  $\# \left( G(J_1+1, \dots, J_d+1) \backslash G(J_1, \dots, J_d) \right) \leq d \left(\frac{K}{N}+1\right) \leq 2dK/N$ we have
\begin{equation} \label{5b}
\# \big( \left\{z_n:~z_n \in A\right\} \backslash G(J_1, \dots, J_d) \big) \leq \frac{2dK}{N},
\end{equation}
and consequently, by~\eqref{(5)},
\begin{eqnarray}
& & \left| \# (\mathcal{Q} \cap A) - \# (\mathcal{Q} \cap G(J_1, \dots, J_d)) \right| \nonumber\\
& \leq & \left| \# (\mathcal{Q} \cap G(J_1+1, \dots, J_d+1)) - \# (\mathcal{Q} \cap G(J_1, \dots, J_d)) \right| \nonumber\\
& \leq & \left| \frac{N}{K} \Big( \# (G(J_1+1, \dots, J_d+1)) - \# (G(J_1, \dots, J_d)) \Big) \right| \\
& & \quad + 40 \sqrt{d} \left(2 + \log_2 N \right)^{(3d+1)/2} + 4 \nonumber\\
& \leq & 2d +  40 \sqrt{d} \left(2 + \log_2 N \right)^{(3d+1)/2} + 4.  \label{5a}
\end{eqnarray}
Thus, by~\eqref{(5)},~\eqref{5b} and~\eqref{5a} we have
\begin{eqnarray*}
& & \left| \sum_{n=1}^N \mathds{1}_A (x_n) - \frac{N}{K} \sum_{n=1}^K \mathds{1}_A (z_n) \right| \\
& = & \left| \# (\mathcal{Q} \cap A) - \frac{N}{K} ~\# \left\{n:~z_n \in A \right\} \right| \\
& \leq & \Big| \# (\mathcal{Q} \cap A) - \# (\mathcal{Q} \cap G(J_1, \dots, J_d)) \Big|  + \left| \# (\mathcal{Q} \cap G(J_1, \dots, J_d)) - \frac{N}{K} \# (G(J_1, \dots, J_d)) \right| \\
& & + \left| \frac{N}{K} \Big(\# (G(J_1, \dots, J_d)) -  \# \left\{n:~z_n \in A\right\} \Big) \right| \\
& \leq & 2d + 40 \sqrt{d} \left(2 + \log_2 N \right)^{(3d+1)/2} + 4 + 20 \sqrt{d} \left(2 + \log_2 N \right)^{(3d+1)/2} + 2 + 2d \\
& \leq & 60 \sqrt{d} \left(2 + \log_2 N \right)^{(3d+1)/2} + 4d + 6.
\end{eqnarray*}
This proves Lemma~\ref{lemmath1}.
\end{proof}

\begin{proof}[Proof of Theorem~\ref{th1}]
All that is left for the proof of Theorem~\ref{th1} is to show that we can closely approximate the measure $\mu$ by a discrete measure of the form $\frac{1}{K} \sum_{n=1}^K \delta_{z_n}$ for some appropriate points $z_1, \dots, z_K$. Such points exist by Corollary~\ref{co1}: For given $N$ we can choose $K = 2^{26} d N^2$ points $z_1, \dots, z_K \in [0,1]^d$ such that
\begin{equation} \label{ratp}
\sup_{A \subset[0,1]^d} \left| \frac{1}{K} \sum_{n=1}^K \mathds{1}_A (z_n) - \mu(A) \right| \leq \frac{1}{N},
\end{equation}
where the supremum is extended over all axis-parallel boxes which have one vertex at the origin. Then by Lemma~\ref{lemmath1} there exist points $x_1, \dots, x_N$ such that
\begin{eqnarray*}
\sup_{A \subset [0,1]^d} \left| \frac{1}{N} \sum_{n=1}^N \mathds{1}_A (x_n) - \mu(A) \right| & \leq & \sup_{A \subset [0,1]^d} \left| \frac{1}{N} \sum_{n=1}^N \mathds{1}_A (x_n) - \frac{1}{K} \sum_{n=1}^K \mathds{1}_A (z_n) \right| + \frac{1}{N} \\
& \leq & \frac{60 \sqrt{d} \left(2 + \log_2 N \right)^{(3d+1)/2} + 4d + 7}{N}\\
& \leq & \frac{63\sqrt{d}\left(2 + \log_2 N \right)^{(3d+1)/2}}{N}.
\end{eqnarray*}
This proves Theorem~\ref{th1}.
\end{proof}

\begin{proof}[Proof of Theorem~\ref{th2}]
Theorem~\ref{th2} follows from Theorem~\ref{th1}. To show that this actually is the case, we use a classical method of constructing finite point sets in the $d+1$-dimensional unit cube in order to obtain a $d$-dimensional infinite sequence. Let $\mu$ be the given measure on $[0,1]^d$, and set $\nu= \mu \times \lambda$, where $\lambda$ is the 1-dimensional Lebesgue measure. Furthermore, for $i \geq 1$ we set 
$$
N_i=2^{(2^i-2)} \qquad \textrm{and} \qquad M_i = \left\{ \begin{array}{ll} 0 & \textrm{for $i=1$} \\ \displaystyle\sum_{l=1}^{i-1} N_l & \textrm{for $i \geq 2$.} \end{array} \right.
$$
By Theorem~\ref{th1} for any $i \geq 1$ there exists an $N_i$-element point set $\mathcal{Q}_i \subset [0,1]^{d+1}$ whose star-discrepancy (with respect to $\nu$) is bounded by
$$
\frac{63 \sqrt{d+1} \left(2 + \log_2 N_i \right)^{(3d+4)/2}}{N_i}.
$$
We write
$$
x^{(i)}_{1}, \dots, x^{(i)}_{N_i}
$$
for the elements of $\mathcal{Q}_i$, rearranged in increasing order according to their last coordinate. Since the order of the points does not influence the discrepancy, we have
\begin{equation} \label{s19e}
D_{N_i}^* (x^{(i)}_{1}, \dots, x^{(i)}_{N_i}; \nu) \leq \frac{63 \sqrt{d+1} \left(2 + \log_2 N_i \right)^{(3d+4)/2}}{N_i}.
\end{equation}
We write $y^{(i)}_n$ for the last coordinate of $x^{(i)}_{n}$. A simple approximation argument shows that we can assume that $y^{(i)}_n$ is strictly increasing for $1 \leq n \leq N_i$.\\

Now let $n \geq 1$ be given. There exists a number $i$ such that $n=M_i + j$ for some $j \in \{1,\dots,N_i\}$. We set
$$
x_n = \bar{x}^{(i)}_j,
$$
where $\bar{x}^{(i)}_j$ is the projection of $x^{(i)}_j$ onto $[0,1]^d$ by omitting its last coordinate. Now let $A$ be any axis-parallel box in $[0,1]^d$ which has one vertex at the origin, and let $N \geq 1$. There exists some $i$ such that $N=M_i + j$ for some $j \in \{1, \dots, N_i\}$. Note that in this case $\log_2 N \geq 2^{i-1}-2$, and consequently $4 + 2 \log_2 N \geq 2^i$. Then by the triangle inequality we have
\begin{eqnarray}
& & \left| \# \{n \leq N:~x_n \in A \} - N \mu (A) \right|  \label{s19a} \\
& \leq & \left| \sum_{l=1}^{i-1} ~\# \{1 \leq n \leq N_l:~ x^{(l)}_n \in A \times [0,1] \} - N_l \nu (A \times [0,1]) \right| \nonumber\\
& & + \left| \# \left\{1 \leq n \leq j:~x^{(i)}_n \in A \times \left[0,y^{(i)}_j\right] \right\} - N_i \nu \left(A \times \left[0,y^{(i)}_j\right] \right) \right| \label{s19b} \\
& & + \left| N_i  \nu \left(A \times \left[0,y^{(i)}_j\right] \right) - j \mu(A) \right| \label{s19c}.
\end{eqnarray}
The terms in lines~\eqref{s19b} and~\eqref{s19c} are each bounded by $N_i D_{N_i}^* (x^{(i)}_{1}, \dots, x^{(i)}_{N_i}; \nu) $. Thus by~\eqref{s19e} we conclude that~\eqref{s19a} is bounded by
\begin{eqnarray*}
& & \left( \sum_{l=1}^{i-1} 63\sqrt{d+1} \left(2 + \log_2 N_l \right)^{(3d+4)/2} \right) + 126 \sqrt{d+1} \left(2 + \log_2 j \right)^{(3d+4)/2}\\
& \leq & 126 \sqrt{d+1} \left( \underbrace{  \frac{1}{2} \left(\sum_{l=1}^{i-1} (2^{l})^{(3d+4)/2} \right)}_{\leq  \frac{(2^{i})^{(3d+4)/2}}{20}} +  (2^{i})^{(3d+4)/2} \right) \\
& \leq & 133 \sqrt{d+1}  (2^{i})^{(3d+4)/2} \\
& \leq & 133 \sqrt{d+1} (4 + 2 \log_2 N)^{(3d+4)/2}.
\end{eqnarray*}
Thus we have
\begin{eqnarray*}
D_N^*(x_1, \dots, x_N; \mu) \leq 133 \sqrt{d+1} \frac{(4 + 2 \log_2 N)^{(3d+4)/2}}{N},
\end{eqnarray*}
which proves Theorem~\ref{th2}.
\end{proof}

\begin{proof}[Proof of Theorem~\ref{th3}]
Theorem~\ref{th3} follows easily from Lemma~\ref{lemmaalex}. In fact, using Lemma~\ref{lemmaalex} for $t= 2^{13} \sqrt{d}$, we immediately see that condition~\eqref{alexcond2} is satisfied. On the other hand, for our choice of $t$ condition~\eqref{alexcond1} is satisfied if
$$
2^{13} \sqrt{d} > \frac{2^{33/2} d}{\sqrt{N}} \log\left(\max\left(\frac{N}{2d},e\right)\right),
$$
which is equivalent to
\begin{equation} \label{equh}
\frac{\sqrt{N}}{\sqrt{d}} > 2^{7/2} \log\left(\max\left(\frac{N}{2d},e\right)\right).
\end{equation}
Equation~\eqref{equh} holds whenever $\sqrt{N} \geq 96 \sqrt{d}$. Thus, using Lemma~\ref{lemmaalex} under the additional assumption that $\sqrt{N}\geq 96 \sqrt{d}$ we get
$$
\p \left( \sup_{C \in \mathcal{C}} |\alpha_N(C)| > 2^{13} \sqrt{d} \right) \leq 16 e^{-2^{26} d} < 1,
$$
which proves the existence of points $t_1, \dots, t_N \in X$ for which
$$
D_N^{(\mathcal{C},P)}(t_1, \dots, t_N) \leq 2^{13} \frac{\sqrt{d}}{\sqrt{N}}.
$$
On the other hand in the case $\sqrt{N} < 96 \sqrt{d}$ the conclusion of Theorem~\ref{th3} holds trivially, since $D_N^{(\mathcal{C},P)}$ is always bounded by 1.
\end{proof}

\begin{proof}[Proof of Theorem~\ref{th4}]
In the case $\Omega=[0,1]^d$ the two discrepancies in~\eqref{discmu} and~\eqref{discomega} coincide, up to an additional multiplicative factor $2^d$ in~\eqref{discomega}, and Theorem~\ref{th4} follows directly from the combination of Theorem~\ref{th1} and Lemma~\ref{lemmabcgt}.\\

Now assume that $\Omega \neq [0,1]^d$. Thus there exists a point $P \in [0,1]^d$ which is \emph{not} contained in $\Omega$. We define the measure $\mu$ as in~\eqref{mua}, that is 
\begin{equation*}
\mu(A) = \frac{\lambda(\Omega \cap A)}{\lambda(\Omega)}
\end{equation*}
for every Borel subset $A$ of $[0,1]^d$. Since by assumption $\lambda(\Omega)>0$, the measure $\mu$ is a non-negative, normalized Borel measure on $[0,1]^d$. Let $\mathcal{B}$ denote the class of Borel subsets of $[0,1]^d$, and define a class $\mathcal{A}$ by
$$
\mathcal{A} = \{ B \cap \Omega:~B \in \mathcal{B} \}.
$$ 
Then the pair $(\Omega,\mathcal{A})$ forms a measurable space. Since $\Omega$ is a measurable subset of $[0,1]^d$, the restriction of $\mu$ to $(\Omega,\mathcal{A})$ is also a measure; with a slight abuse of notation we will denote this measure also by $\mu$. By definition $\mu(\Omega)=1$, and thus $(\Omega,\mathcal{A},\mu)$ is a probability space.\\

As mentioned in Section~\ref{sec2}, it is well-known that the class of axis-parallel boxes which are contained in $[0,1]^d$ and have vertex at the origin is a VC-class of dimension $d$. Let $\mathcal{C}$ denote the class of all sets which can be obtained by intersecting such an axis-parallel box with $\Omega$. It is an easy exercise to check that then the class $\mathcal{C}$ is also a VC-class, of dimension at most $d$. Applying Theorem~\ref{th3} to the probability space $(\Omega,\mathcal{A},\mu)$ and the class $\mathcal{C}$ this implies that there exist points $z_1, \dots, z_K \in \Omega$ for which
$$
D_N^*(z_1, \dots, z_K;\mu) = D_N^{(\mathcal{C},\mu)} (z_1, \dots, z_K) \leq 2^{13} \frac{\sqrt{d}}{\sqrt{N}}
$$
(the class $\mathcal{C}$ is not countable, but this problem can be overcome in the way described before the statement of Corollary~\ref{co1}). It is important that in this approximation \emph{the points $z_1, \dots, z_K$ are elements of $\Omega$}. Taking $K = 2^{26} d N^2$ we get the points $z_1, \dots, z_K$ for the proof of Theorem~\ref{th1}, which satisfy~\eqref{ratp}. The proof of Lemma~\ref{lemma2} reveals that the points $x_1, \dots, x_N$ in the conclusion of Theorem~\ref{th1} are chosen from $z_1, \dots, z_K$. Consequently, since in the present case we have $z_1, \dots, z_K \in \Omega$, in this case the points from the conclusion of Theorem~\ref{th1} also satisfy $x_1, \dots, x_N \in \Omega$.\\

For these points we have
$$
D_N^*(x_1, \dots, x_N;\mu) \leq \frac{63 \sqrt{d} \left(2 + \log_2 N \right)^{(3d+1)/2}}{N}.
$$
Let 
$$
M = \left\lceil \frac{N}{\lambda(\Omega)} \right\rceil;
$$
then we have
\begin{equation} \label{emm}
\left| \frac{1}{M} - \frac{\lambda(\Omega)}{N}\right| \leq \frac{\lambda(\Omega)^2}{N^2} \leq \frac{\lambda(\Omega)}{N^2}.
\end{equation}
We define points 
$$
x_{N+1} = \dots =  x_{M} = P
$$
(where $P$ is the point for which $P \not\in \Omega$). Note that for the discrepancy defined in~\eqref{discomega} the supremum is taken over all axis-parallel boxes contained in $[0,1]^d$, without the additional assumption that one corner of each such box has be at the origin as in the definition of the discrepancy in~\eqref{discmu}. Restricting the supremum to only those axis-parallel boxes which have a vertex at the origin will add a multiplicative factor $2^d$. Thus, using~\eqref{emm}, we get
\begin{eqnarray*}
D_M^{(\Omega)}(x_1, \dots, x_M) & \leq & 4^d \sup_{A \subset [0,1]^d} \left| \frac{1}{M} \sum_{n=1}^M \mathds{1}_{\Omega \cap A} (x_n) - \lambda(\Omega \cap A) \right| \\
& = & 4^d \sup_{A \subset [0,1]^d} \left| \frac{1}{M} \sum_{n=1}^N \mathds{1}_{A} (x_n) - \mu(A) \lambda(\Omega)  \right| \\
& \leq & 4^d \left(\frac{\lambda(\Omega)}{N} + \sup_{A \subset [0,1]^d} \left| \frac{\lambda(\Omega)}{N} \sum_{n=1}^N \mathds{1}_{A} (x_n) - \mu(A) \lambda(\Omega)  \right| \right) \\
& = & 4^d \lambda(\Omega)  \left(\frac{1}{N} + D_N^*(x_1, \dots, x_N; \mu) \right) \\
& \leq & 4^d \lambda(\Omega)~\frac{63 \sqrt{d} \left(2 + \log_2 N \right)^{(3d+1)/2} + 1}{N},
\end{eqnarray*}
where the supremum is always taken over those axis-parallel boxes $A$ which have a vertex at the origin. Now by Lemma~\ref{lemmabcgt} we have, using again~\eqref{emm},
\begin{eqnarray*}
& & \left| \frac{1}{N} \sum_{n=1}^N f (x_n) - \frac{1}{\lambda(\Omega)} \int_\Omega f(x)~d\lambda(x) \right| \\
& \leq & \frac{1}{\lambda(\Omega)} \left| \frac{1}{M} \sum_{n=1}^M (f \cdot \mathds{1}_\Omega) (x_n) - \int_\Omega f(x)~d\lambda(x) \right| + \frac{\|f\|_\infty}{N} \\
& \leq & \frac{1}{\lambda(\Omega)} D_M^{(\Omega)}(x_1, \dots, x_M) \cdot V(f) + \frac{\|f\|_\infty}{N} \\
& \leq & 4^d ~\frac{63 \sqrt{d} \left(2 + \log_2 N \right)^{(3d+1)/2} + 1}{N} + \frac{\|f\|_\infty}{N} ,
\end{eqnarray*}
which proves Theorem~\ref{th4}.
\end{proof}

 \section*{Acknowledgment}

We want to thank Richard Nickl for several helpful remarks concerning empirical processes indexed by sets and the metric entropy of VC classes, and Benjamin Doerr for valuable explanations concerning the discrepancies of hypergraphs and matrices.

%\bibliography{Low_discr_nonunif}
%\bibliographystyle{abbrv}

\end{document}